\documentclass{article}
\usepackage{fix-cm}
\newcommand{\Helia}[1]{\textcolor{purple}{Helia: #1}}
\newcommand{\Cam}[1]{\textcolor{red}{Cameron: #1}}

\usepackage{graphicx} 

\usepackage{comment}
\usepackage{wrapfig} 
\usepackage{xcolor}
\usepackage{xspace}
\usepackage{graphicx}
\usepackage{subcaption}
\usepackage[T1]{fontenc}
\usepackage{lmodern}
\usepackage{amsthm} 
\usepackage{algorithm}
\usepackage{algpseudocode}
\usepackage[utf8]{inputenc}
\usepackage{parskip}
\usepackage{verbatim}
\usepackage{tikz}

\theoremstyle{plain}
\newtheorem{theorem}{Theorem}

\newtheorem{lemma}[theorem]{Lemma}

\theoremstyle{definition}

\usepackage{hyperref}
\usepackage{amsmath, amsfonts, amssymb}
\usepackage{cleveref}
\begin{document}

\title{A Combinatorial Proof of Cayley’s Formula via Degree Sequences}
\markright{Abbreviated Article Title}
\author{
Helia Karisani \quad \quad Mohammadreza Daneshvaramoli \\
\makebox[\textwidth]{\normalfont University of Massachusetts Amherst} \\
\makebox[\textwidth]{\normalfont \{hkarisani, mdaneshvaram\}@umass.edu}
}
\maketitle
\begin{abstract}
Cayley’s formula is a fundamental result in combinatorics, counting the number of a labeled tree with certain number of nodes. While existing proofs employ approaches such as Prüfer sequences and the Matrix-Tree Theorem, we introduce a novel combinatorial proof that provides an intuitive perspective on tree enumeration, shedding light on the relationships between degree sequences and the structural properties of labeled trees. We aim to make this proof accessible and insightful, offering potential applications to related combinatorial problems.
\end{abstract}
\vspace{-1em}
\section{Introduction}
Cayley’s formula is a foundational result in combinatorics, stating that the number of labeled trees on \(n\) vertices is \(n^{n-2}\). First conjectured by Borchardt in 1860~\cite{Borchardt1860} and later proved by Cayley in 1889~\cite{Cayley1889}, the theorem has inspired numerous elegant proofs, including those based on Prüfer sequences~\cite{prufer1918}, the Matrix-Tree Theorem, and various bijective constructions~\cite{aigner2010proofs}
\section{Cayley’s formula Proof}
\label{sec:cayleys_proof}
We begin with a crucial theorem on counting trees with specified degrees, followed by our proof of Cayley’s formula.
\begin{theorem}
	\label{th:d1}
	Suppose $n \geq 2$ and $d_1, d_2, \dots, d_n$ are positive natural numbers with a sum of $2n-2$. In this case, the number of trees with vertices $\{1, 2, \dots, n\}$ and with vertex degrees $d_1, d_2, \dots, d_n$ is given by:
	\begin{equation}
		T_{n,d_1,d_2,\dots,d_n} = \frac{(n - 2)!}{(d_1-1)!(d_2-1)!\dots(d_n-1)!}.
	\end{equation}
\end{theorem}
\begin{proof}
For the proof, please refer to \cite{joyal1981}, by Joyal and André.
\end{proof}	
Below, Cayley's formula is stated along with our proof.
\begin{theorem} 
	\label{th:ca1}
	\textbf{Cayley's Formula:} The number of trees that can be formed by $n$ labeled vertices is $n^{n-2}$. 
\end{theorem}
\begin{proof}
To prove Theorem \ref{th:ca1}, we use induction on the theorem statement.\smallskip

\noindent \textbf{Base Case:} For $n = 2$, exactly one tree can be formed, i.e., $T_1 = 2^0 = 1$.
	
\noindent \textbf{Inductive Hypothesis:} Now assume the theorem holds for all values $1, 2, \dots, n-1$. 
    \begin{align*}
	\forall i \in \{2, \dots, n-1\} :T_i = i^{i-2}.
    \end{align*}	
\noindent\textbf{Inductive Step:} We want to prove the statement for $n$: $T_n = n^{n-2}$.	
    We begin by analyzing the role of the first vertex, $v_1$, whose degree $\deg(v_1) = k$ for $k \in\{ 1,\ldots,n-1\}$. This allows us to categorize trees based on the degree of $v_1$, leading to the expression:
	\begin{equation}
		\label{eq:e16}
		T_n = \sum_{k=1}^{n-1} \#\text{labeled trees where } \deg(v_1) = k
	\end{equation}
    Removing vertex~1 splits the graph into \( k \) connected subtrees with sizes \( a_1, \ldots, a_k \), satisfying \( \sum_{i=1}^k a_i = n-1 \), as shown in Figure~\ref{fig:f1}. 
\begin{figure}[htbp]
\centering
\begin{minipage}[t]{0.48\textwidth}
    \centering
    \includegraphics[width=0.95\linewidth]{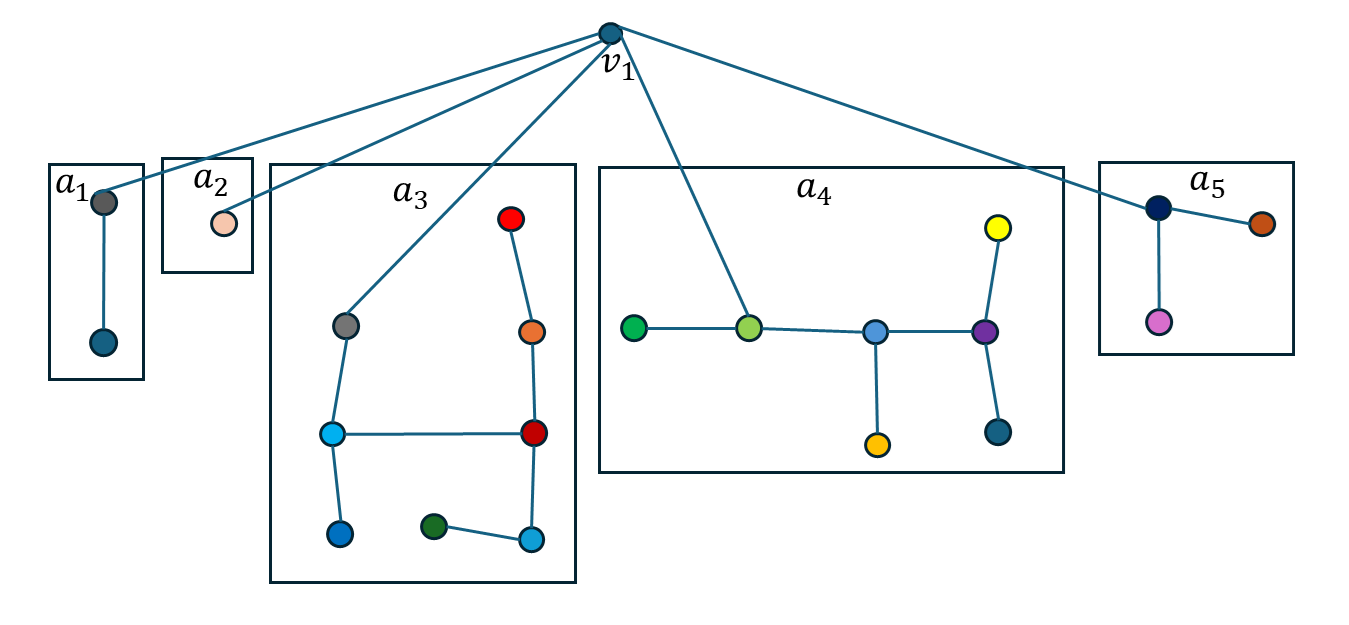}
    \caption{A tree with \( d_1 = k = 5 \) and connected subtrees \( \{a_i\} \).}
    \label{fig:f1}
\end{minipage}
\hfill
\begin{minipage}[t]{0.48\textwidth}
    \centering
    \includegraphics[width=0.95\linewidth]{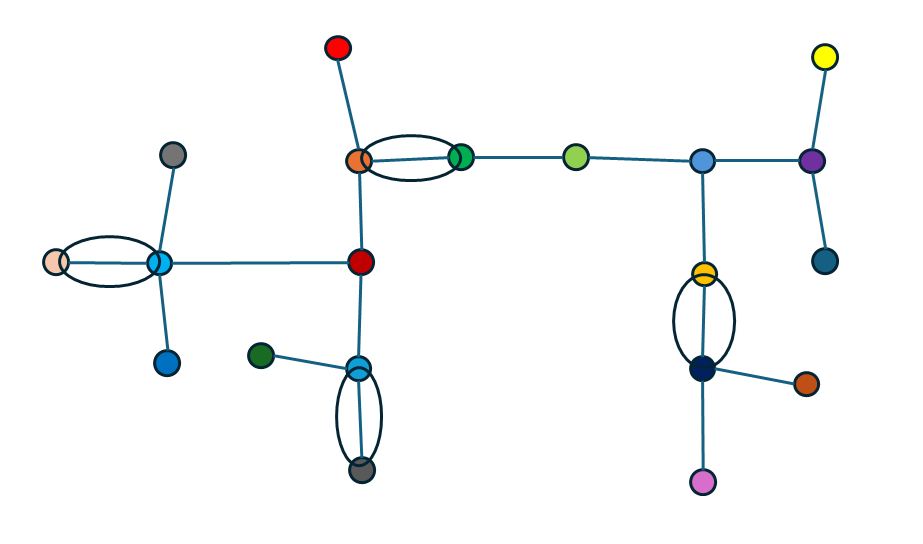}
    \caption{One way of choosing $k-1 = 4$ specified edges (circled) in a tree.}
    \label{fig:f2}
\end{minipage}
\end{figure}
Therefore, the total count is obtained by summing over all configurations of \( a_1 \) to \( a_k \), representing all ways of partitioning the \( n - 1 \) remaining vertices into \( k \) connected components, such that \( \sum_{i=1}^k a_i = n - 1 \). To compute this, we consider each possible multiset of \( k \) positive integers summing to \( n - 1 \), and assign the variables \( a_1, a_2, \ldots, a_k \) to the values in the multiset. A key observation is that the ordering of the labels \( a_1, \ldots, a_k \) does not affect the configuration; e.g., \( a_1 = 1, a_2 = 2 \) and \( a_1 = 2, a_2 = 1 \) represent the same structure. So, since different orderings of the same multiset represent the same structure, we divide by \( k! \) to avoid overcounting these indistinguishable configurations that happens due to index permutations.
\begin{equation}
\label{eq:e17}
\#\text{Trees where} \deg(v_1) \text{ is } k =
\frac{
\sum\limits_{\{a_1, \ldots, a_k\} : \sum a_i = n - 1}
\text{\#trees for fixed multiset } \{a_1, \ldots, a_k\}
}{k!}.
\end{equation}
In the next step, given a fixed multiset \( \{a_1, \ldots, a_k\} \), representing the sizes of the \( k \) subtrees, the remaining \( n - 1 \) vertices must be divided into \( k \) disjoint groups of sizes \( a_1, \ldots, a_k \). This corresponds to choosing subsets of those sizes from the labeled vertex set, as computed below:
	\begin{align}
            \label{eq:divide}
		\#\text{Ways of dividing $n-1$ vertices into $k$ fixed-sized graphs} =  \frac{(n-1)!}{\prod_{i=1}^{k} a_i!}.
	\end{align}
	Furthermore, in each component, one vertex must be connected to vertex \( 1 \). Thus, we select one vertex from each component, which can be done in \( \prod_{i=1}^{k} a_i \) ways. Within each component, it is known that a tree structure exists, and there are \( T_{a_i} \) distinct trees for a component of size \( a_i \). Consequently, for each fixed set of \( a_i \) values satisfying \( \sum_{i=1}^k a_i = n-1 \), the total number of trees is given by:
	\begin{equation}
		\label{eq:e19}
		\#\textit{Trees with fixed }a_i\textit{ size} =\left( \prod_{i=1}^{k} a_i T_{a_i} \right) \frac{(n-1)!}{\prod_{i=1}^{k} a_i!} .
	\end{equation}
This result is obtained by combining the number of ways to partition the vertices, the number of ways to select a connection to vertex \(1\), and the number of distinct subtrees. We choose to simplify by merging Equations \eqref{eq:divide}, \eqref{eq:e17}, and \eqref{eq:e19}, we can derive a formula for Equarion~\ref{eq:e16}:
	\begin{equation}
    \label{eq:eq20}
		T_n = \sum_{k=1}^{n-1} \frac{\displaystyle \sum_{\{a_1,\ldots,a_k\}: \sum a_i = n-1}  \left( \prod_{i=1}^{k} a_i T_{a_i} \right) \frac{(n-1)!}{\prod_{i=1}^{k} a_i!} }{k!}.
	\end{equation}
Below, we demonstrate that expressing \( T_n \) in terms of the degree of vertex~1 leads to a remarkably simple identity.
\begin{lemma}
\label{lem:a1}
For all integers \( n \geq 2 \) and for all integers \( k \) such that \( 1 \leq k \leq n-1 \), the following identity holds:
    \begin{equation}
    \frac{\displaystyle \sum_{\{a_1,\ldots,a_k\}: \sum a_i = n-1}  \left( \prod_{i=1}^{k} a_i T_{a_i} \right) \frac{(n-1)!}{\prod_{i=1}^{k} a_i!} }{k!} = \frac{T_{n-1}}{(n-1)^{k-2}} \binom{n-2}{k-1}.
    \end{equation}
\end{lemma}
	\begin{proof}
    To prove this, we multiply both sides by \( (n-1)^{k-2} \) and set \( n-1 = m \) for clarity. It suffices to show that the right hand side \textit{(RHS)} equals the left hand side \textit{(LHS)}.
		\begin{equation}
        \label{eq:premoreda}
			\text{LHS} = \frac{\displaystyle \sum_{\{a_1,\ldots,a_k\}: \sum a_i = n-1}  \left( \prod_{i=1}^{k} a_i T_{a_i} \right) \frac{m!}{\prod_{i=1}^{k} a_i!} }{k!} m^{k-2} = T_{m} \binom{m-1}{k-1}  = \text{RHS}.
		\end{equation}
To prove the above statement, we use a double-counting argument. On the right-hand side, \( T_m \) counts all labeled trees on \( m \) vertices. Each such tree has exactly \( m - 1 \) edges, and removing \( k - 1 \) of them splits the tree into \( k \) connected components. There are \( \binom{m - 1}{k - 1} \) ways to choose which \( k - 1 \) edges to remove, and an example of such a selection is illustrated in Figure~\ref{fig:f2}. Thus, the right-hand side counts labeled trees on \( m \) vertices together with a choice of \( k - 1 \) specified edges whose removal results in a forest of \( k \) components. It remains to show that the left-hand side counts the same set of structures, which we do by analyzing the count from a component-based perspective.
		\begin{align}
    \text{LHS} = \frac{1}{k!} \left( \sum_{\sum_{i=1}^{k} a_i = m} L_1 L_2 L_3 \right),
    \, L_1 = \frac{m!}{\prod_{i=1}^{k} a_i!}, \,
    L_2 = \prod_{i=1}^{k} T_{a_i},\,
    L_3 = m^{k-2} \prod_{i=1}^{k} a_i. \label{eq:main_lhs}
\end{align}
		First, let us examine \( L_1 \). $L_1$ represents the number of ways to partition \( m \) distinct vertices into \( k \) groups, each containing \( a_i \) vertices, similar to Equation~\eqref{eq:divide}. Next, we consider the meaning of \( L_2 \). Within each group, we construct a connected tree. Given that each group contains \( a_1, a_2, \dots, a_k \) vertices, the number of possible trees in a group of size \( a_i \) is \( T_{a_i} \). Therefore, \( L_2 \) represents the total number of ways to form trees within each of the groups.

Finally, we analyze \( L_3 \). Note that for any positive integer \( r \) and non-negative integers \( c_1, c_2, \ldots, c_k \) such that \( c_1 + c_2 + \cdots + c_k = r \), the multinomial identity states that $(a_1 + a_2 + \cdots + a_k)^r = \sum_{c_1 + \cdots + c_k = r} \frac{r!}{c_1! \cdots c_k!} \prod_{i=1}^k a_i^{c_i}$. Using that \( \sum a_i = m \), we can write \( m^{k-2} = \left( \sum_{i=1}^{k} a_i \right)^{k-2} \) and expanding out $L_3$ using multinomial identity we get: 
\begin{equation}
\label{eq:L_3}
\begin{split}
L_3 = m^{k-2} \prod_{i=1}^k a_i 
= \left( \sum_{i=1}^k a_i \right)^{k-2} \prod_{i=1}^k a_i 
\overset{\text{(a)}}{=}
\sum_{\substack{c_1 + \cdots + c_k = k - 2}} 
\frac{(k - 2)!}{\prod_{i=1}^k c_i!} \prod_{i=1}^k a_i^{c_i + 1}.
\end{split}
\end{equation}
We applied the multinomial identity of Equation() in step (a) of Equation~\ref{eq:L_3}.

Observe that in the multinomial expansion, each \( a_i \) appears \( c_i \geq 0 \) times, and the extra factor \( \prod_{i=1}^k a_i \) results in each \( a_i \) being raised to the power \( c_i + 1 \). Letting \( d_i = c_i + 1 \), we then have \( \sum_{i=1}^k d_i = 2k - 2 \). This transforms the exponent expression into \( \prod_{i=1}^k a_i^{d_i} \), and the multinomial coefficient becomes \( \frac{(k - 2)!}{\prod_{i=1}^k (d_i - 1)!} \). According to Theorem~\ref{th:d1}, this quantity counts the number of labeled trees on \( k \) vertices with degree sequence \( (d_1, \dots, d_k) \).

Now, interpret the $k$ groups of size $a_i$ (formed in $L_1$) as distinct components, where each component is treated as a "super vertex." These $k$ super vertices are then connected to form a single tree. The number of such trees, where the degree of each super vertex is $d_i$, is given by $\frac{(k-2)!}{\prod_{i=1}^k c_i!}$, as in Theorem~\ref{th:d1}. In this interpretation, $d_i$ represents the number of edges leaving component $i$ to connect it to the rest of the tree.


For each component of size \(a_i\), there are \(a_i^{d_i}\) ways to assign its \(d_i\) outgoing edges to its vertices, since each edge can start from any of the \(a_i\) nodes. Therefore, the total number of edge assignment configurations across all \(k\) components is \(\prod_{i=1}^k a_i^{d_i}\). The number of labeled trees with degree sequence \((d_1, \ldots, d_k)\) is \((k - 2)! / \prod_{i=1}^k (d_i - 1)!\), and summing over all valid degree sequences with \(\sum d_i = 2k - 2\), we obtain \(L_3 = \sum_{\{d_1, \ldots, d_k\} : \sum d_i = 2k - 2} \frac{(k - 2)!}{\prod_{i=1}^k (d_i - 1)!} \prod_{i=1}^k a_i^{d_i}\). This expression counts the total number of ways to connect \(k\) components into a single tree, accounting for all valid ways of assigning edge directions from each component.	

        Thus, \(L_1 \times L_2 \times L_3\) represents the total number of ways to (1) partition the \(n - 1\) vertices into \(k\) groups of sizes \(a_1, \ldots, a_k\), (2) form a labeled tree within each group, and (3) connect the groups using \(k - 1\) additional edges. Each resulting configuration corresponds to a tree where the \(k\) groups behave like "super vertices" connected by the specified edges. This is equivalent to specifying \(k - 1\) edges in a tree, since one can remove these edges to recover the original \(k\)-component structure. Figure~\ref{fig:f3} illustrates the equivalence to the earlier construction in Figure~\ref{fig:f2}.
Finally, each such tree is counted \(k!\) times due to the possible permutations of the \(k\) component labels. To correct for this overcounting, we divide by \(k!\), and the total count becomes \(\sum L_1 \times L_2 \times L_3 / k! = RHS\).
	\begin{figure}[htbp]
    \centering
    \includegraphics[width=0.4\textwidth]{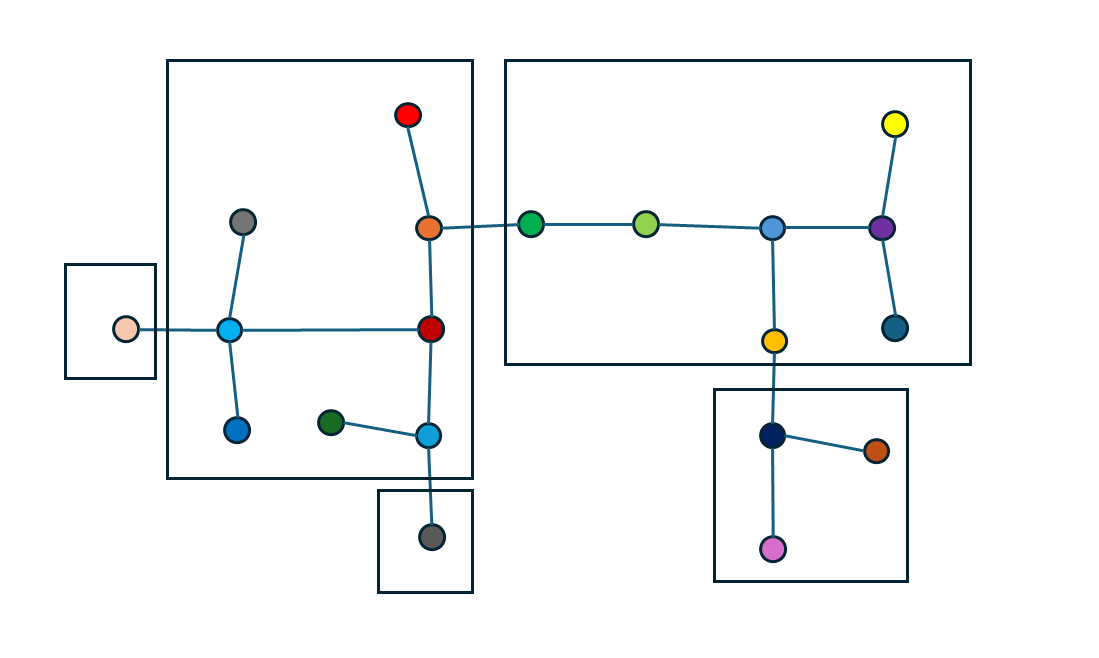}
    \caption{\(k=5\) tree components of Figure~\ref{fig:f2}.}
    \label{fig:f3}
\end{figure}
	\end{proof}
    Back to proving Theorem~\ref{th:ca1}, we rewrite $T_n$,
\begin{align}
    T_n = \sum_{k=1}^{n-1} \frac{T_{n-1}}{(n-1)^{k-2}} \binom{n-2}{k-1} 
    \quad \text{(from Lemma~\ref{lem:a1} and Equation~\ref{eq:eq20})}.
    \label{eq:rewrite_Tn}
\end{align}
    \begin{align}
    T_n = \sum_{k=1}^{n-1} (n-1)^{n-1-k} \binom{n-2}{k-1} \quad \text{(from induction hypothesis)}.
    \label{eq:induction}
\end{align}

Substituting \( j = k - 1 \) into Equation~\ref{eq:induction}, simplifies the summation: \( T_n = \sum_{j=0}^{n-2} (n - 1)^{n - 2 - j} \binom{n - 2}{j} \). Now we apply Newton’s Binomial Theorem \((x + y)^r = \sum_{j = 0}^r \binom{r}{j} x^{r - j} y^j\), by setting \(r = n - 2\), \(x = n - 1\), and \(y = 1\), yielding the identity
\[
\sum_{j=0}^{n-2} \binom{n-2}{j} (n-1)^{n-2-j} = ((n-1)+1)^{n-2}.
\]
Thus, we conclude $T_n = n^{n-2}$.
\end{proof}
\section{Conclusion} This proof verifies Cayley’s Theorem by combining induction and degree-based decomposition. By partitioning trees using vertex degrees and reassembling them as super-nodes, we provide a new combinatorial explanation that is both intuitive and rigorous, offering insight beyond classical approaches like Prüfer sequences or matrix methods.
\\

\vfill\eject

\end{document}